\documentclass[reqno,11pt]{amsart}

\usepackage[all]{xy}
\usepackage{amssymb}
\usepackage{amsgen}
\usepackage{amsmath}
\usepackage{amsthm}
\usepackage{mathrsfs}
\usepackage{cite}
\usepackage{amsfonts}

\newcommand{\Bis}{\Gamma}
\newcommand{\skel}[1]{^{(#1)}}
\newcommand{\Is}{\mathop{\mathrm{Is}}}
\newcommand{\dom}{\mathop{\boldsymbol d}}
\newcommand{\ran}{\mathop{\boldsymbol r}}

\renewcommand{\to}{\longrightarrow}

\newcommand{\inv}{^{-1}}
\newcommand{\p}{\varphi}

\newcommand{\ov}[1]{\ensuremath{\overline {#1}}}
\newcommand{\til}[1]{\ensuremath{\widetilde {#1}}}

\newcommand{\wh}{\widehat}

\newcommand{\sour}{\mathop{\boldsymbol s}}
\usepackage{xcolor}

\newtheorem{Thm}{Theorem}[section]
\newtheorem{Prop}[Thm]{Proposition}

\newtheorem{Lemma}[Thm]{Lemma}
{\theoremstyle{definition}
}
{\theoremstyle{remark}
}
\newtheorem{Cor}[Thm]{Corollary}
{\theoremstyle{remark}
}
{\theoremstyle{remark}
}

{\theoremstyle{remark}
}
{\theoremstyle{remark}
}
{\theoremstyle{remark}
}

\numberwithin{equation}{section}

\title[Prime  \'etale groupoid algebras]{Prime \'etale groupoid algebras with applications to inverse semigroup and Leavitt path algebras}

\author{Benjamin Steinberg}
\address[B.~Steinberg]{%
    Department of Mathematics\\
    City College of New York\\
    Convent Avenue at 138th Street\\
    New York, New York 10031\\
    USA}
\email{bsteinberg@ccny.cuny.edu}

\thanks{The author was supported by  NSA MSP \#H98230-16-1-0047.}
\date{February 5, 2018}


\begin{document}

\begin{abstract}
In this paper we give some sufficient and some necessary conditions for an \'etale groupoid algebra to be a prime ring.  As an application we recover the known primeness results for inverse semigroup algebras and Leavitt path algebras. It turns out that primeness of the algebra is connected with the dynamical property of topological transitivity of the groupoid.  We obtain analogous results for semiprimeness.
\end{abstract}

\maketitle

\section{Introduction}
The author introduced in~\cite{mygroupoidalgebra} a purely algebraic analogue of \'etale groupoid $C^*$-algebras~\cite{Renault,Paterson}, both to serve as a unifying factor between the theories of Leavitt path algebras~\cite{LeavittBook} and graph $C^*$-algebras~\cite{cuntzkrieger} and to provide a new approach to inverse semigroup algebras.   In recent years, there has been a lot of work around \'etale groupoid algebras~\cite{operatorsimple1,operatorguys2,reconstruct,GroupoidMorita,CarlsenSteinberg,Strongeffective,ClarkPardoSteinberg,groupoidbundles,groupoidprimitive,Clarkdecomp,GonRoy2017a,GonRoy2017,Hazrat2017,AraSims2017,Purelysimple,gpdchain}.

The groupoid and inverse semigroup interaction has been, in this author's opinion, a two-way street.  The simplicity criteria for algebras of Hausdorff groupoids~\cite{operatorsimple1,operatorguys2} enabled the author~\cite{groupoidprimitive} to make progress on an old question of Douglas Munn about which contracted inverse semigroup algebras are simple~\cite{MunnAlgebraSurvey}.  On the other hand, inspired by work of Domanov~\cite{Domanov} and Munn~\cite{Munnprimitive,MunnSemiprim,MunnSemiprim2} for inverse semigroup algebras, the author made significant progress toward the study of primitivity and semiprimitivity of groupoid algebras~\cite{groupoidprimitive}.  In particular, the primitivity results for Leavitt path algebras~\cite{primleavittbell} were given a more conceptual explanation.

In this paper, we turn to the related question of when a groupoid algebra is a prime or semiprime ring;  recall that a ring is prime if $0$ is a prime ideal and it is semiprime if it has no nilpotent ideals.  Of course (semi)primitive rings are (semi)prime, but the converse is not true.  In~\cite{primleavittbell} it was observed that, for countable graphs satisfying condition (L), primitivity and primeness were equivalent but that things changed for infinite graphs.  This too deserves a conceptual explanation.

Inspired by results of Munn~\cite{Munnprime} we develop some necessary conditions and some sufficient conditions for an \'etale groupoid algebra to be a prime or semiprime ring.  Note that group algebras are particular cases of \'etale groupoid algebras and primeness and semiprimeness for group algebras was characterized long ago~\cite{connell,PassmanBook}.  Our conditions are general enough to recover all the known results for both inverse semigroup algebras and for Leavitt path algebras.  And, as is often the case, the proofs for groupoids are easier than the original inverse semigroup arguments.

It turns out that both primeness and primitivity are related to dynamical properties of the groupoid.  Generalizing the notion of topological transitivity from dynamical systems, we say that an \'etale groupoid is topologically transitive if each non-empty open invariant subspace of the unit space is dense.    If the groupoid has a dense orbit, then it is topologically transitive; the converse holds for second countable groupoids by a Baire category argument.  For an effective Hausdorff groupoid, it turns out that primitivity of its algebra over a field is equivalent to having a dense orbit~\cite{groupoidprimitive} and it is shown here that primeness is equivalent to topological transitivity.  In particular, for second countable Hausdorff effective groupoids, primitivity and primeness are equivalent.  This explains the results of~\cite{primleavittbell} because a Leavitt path algebra satisfies condition (L) precisely when the associated groupoid is effective and the groupoid is second countable when the graph is countable.

For more general groupoids, topological transitivity is not enough for primeness.  We can show that if there is a dense orbit whose isotropy group has a prime group algebra (i.e., has no non-trivial normal subgroup of finite order), then the groupoid algebra (over an integral domain) is prime.  This generalizes Munn's result for ($0$-)bisimple inverse semigroup algebras~\cite{Munnprime}.

It is an open question to characterize completely (semi)prime \'etale group\-oid algebras (or even inverse semigroups algebras).

The paper is organized as follows.  We begin with a section on groupoids, inverse semigroups and their algebras.  In the next section we define topological transitivity for \'etale groupoids and establish some basic results concerning the notion.  The following section turns to our necessary and our sufficient conditions for primeness and semiprimeness of groupoid algebras (unfortunately the conditions do not coincide).  The final section recovers the results of~\cite{primleavittbell} characterizing prime and primitive Leavitt path algebras from the more general groupoid results and recovers the results of Munn~\cite{Munnprime} for inverse semigroup algebras.

\section{Groupoids, inverse semigroups and their algebras}
This section contains preliminaries about groupoids, inverse semigroups and their algebras.  Lawson~\cite{Lawson} is our recommended reference for inverse semigroup theory. For \'etale groupoids, we recommend~\cite{Renault,Exel,Paterson}.  Algebras of ample groupoids were introduced in~\cite{mygroupoidalgebra}; see also~\cite{mygroupoidarxiv} for some additional results not included in~\cite{mygroupoidalgebra}, as well as~\cite{operatorguys1} where the notion was introduced independently.

\subsection{Inverse semigroups}
An \emph{inverse semigroup} is a semigroup $S$ such that, for all $s\in S$, there exists unique $s^*\in S$ with $ss^*s=s$ and $s^*ss^*=s^*$.  Notice that $s^*s,ss^*$ are idempotents. Also, note that $(st)^*=t^*s^*$.  Idempotents of $S$ commute and so $E(S)$ is a subsemigroup.  Moreover, it is a meet semilattice with respect to the ordering $e\leq f$ if $ef=e$.  In fact, $S$ itself is ordered by $s\leq t$ if $s=te$ for some idempotent $e\in E(S)$ or, equivalently, $s=ft$ for some $f\in E(S)$.  This partial order is compatible with multiplication and stable under the involution.
If $e\in E(S)$, then $G_e=\{s\in S\mid s^*s=e=ss^*\}$ is a group called the \emph{maximal subgroup} of $S$ at $e$.  It is the group of units of the monoid $eSe$.

All groups are inverse semigroups, as are all (meet) semilattices. If $X$ is a topological space, then the set of all homeomorphisms between open subsets of $X$ is an inverse semigroup $I_X$ under the usual composition of partial functions.  An inverse semigroup $S$ has a \emph{zero element} $z$, if $zs=z=sz$ for all $s\in S$.  Zero elements are unique when they exist and will often be denoted by $0$.  The zero element of $I_X$ is the empty partial bijection.

By an action of an inverse semigroup $S$ on a space $X$, we mean a homomorphism $\theta\colon S\to I_X$ such that if we put $X_e=\mathrm{dom}(\theta(e))$, then  \[\bigcup_{e\in E(S)}X_e= X.\]  This last condition
is a non-degeneracy condition and implies, for instance, that a group must act by homeomorphisms.

If $R$ is a commutative ring with unit, then the \emph{semigroup algebra} $RS$ of an inverse semigroup $S$ is defined as the $R$-algebra with basis $S$ and multiplication extending that of $S$ via the distributive law. If $S$ is an inverse semigroup with zero element $z$, then the \emph{contracted semigroup algebra} is $R_0S=RS/Rz$.  The contracted semigroup algebra construction amounts to amalgamating the zero of $S$ with the zero of $R$ and it is universal for zero-preserving representations of $S$ into $R$-algebras.


\subsection{\'Etale groupoids}
In this paper, following Bourbaki, compactness will include the Hausdorff axiom.  However, we do not require locally compact spaces to be Hausdorff. A topological groupoid $\mathscr G=(\mathscr G\skel 0,\mathscr G\skel 1)$ is \emph{\'etale} if its domain map $\dom$ (or, equivalently, its range map $\ran$) is a local homeomorphism.  In this case, identifying objects with identity arrows, we have that $\mathscr G\skel 0$ is an open subspace of $\mathscr G\skel 1$ and the multiplication map is a local homeomorphism.  Details can be found in~\cite{Paterson,resendeetale,Exel}.

Following~\cite{Paterson}, an \'etale groupoid is called \emph{ample} if its unit space $\mathscr G\skel 0$ is locally compact Hausdorff with a basis of compact open subsets. We shall say that an ample groupoid $\mathscr G$ is Hausdorff if $\mathscr G\skel 1$ is Hausdorff.

A \emph{local bisection} of an \'etale groupoid $\mathscr G$ is an open subset $U\subseteq \mathscr G\skel 1$ such that both $\dom|_U$ and $\ran|_U$ are homeomorphisms.  The local bisections form a basis for the topology on $\mathscr G\skel 1$~\cite{Exel}. The set $\Bis(\mathscr G)$ of local bisections is an inverse monoid under the binary operation \[UV = \{uv\mid u\in U,\ v\in V,\ \dom (u)=\ran (v)\}.\] The semigroup inverse is given by $U^* = \{u\inv\mid u\in U\}$ and $E(\Bis(\mathscr G))=\Bis(\mathscr G\skel 0)$. The inverse monoid $\Bis(\mathscr G)$ acts on $\mathscr G\skel 0$ by partial homeomorphisms by putting \[U\cdot x=\begin{cases}y, & \text{if there is $g\in U$ with $\dom(g)=x,\ran(g)=y$}\\ \text{undefined},  & \text{else.}\end{cases}\] The set $\Bis_c(\mathscr G)$ of compact local bisections  is an inverse subsemigroup of $\Bis(\mathscr G)$ (it is a submonoid if and only if $\mathscr G\skel 0$ is compact)~\cite{Paterson}. Note that $\mathscr G$ is ample if and only if $\Bis_c(\mathscr G)$ is a basis for the topology on $\mathscr G\skel 1$~\cite{Exel,Paterson}.

The \emph{isotropy subgroupoid} of a groupoid $\mathscr G=(\mathscr G\skel 0,\mathscr G\skel 1)$ is the subgroupoid $\Is(\mathscr G)$ with $\Is(\mathscr G)\skel 0=\mathscr G\skel 0$ and \[\Is(\mathscr G)\skel 1=\{g\in \mathscr G\skel 1\mid \dom(g)=\ran(g)\}.\]  The \emph{isotropy group} of $x\in \mathscr G\skel 0$ is the group \[G_x=\{g\in \mathscr G\skel 1\mid \dom(g)=x=\ran(g)\}.\] An \'etale groupoid is said to be \emph{effective} if $\mathscr G\skel 0=\mathrm{Int}(\Is(\mathscr G)\skel 1)$, the interior of the isotropy bundle.  It is well known, and easy to prove, that an ample groupoid $\mathscr G$ is effective if and only if the natural action of $\Bis_c(\mathscr G)$ on $\mathscr G\skel 0$ is faithful.

If $x\in \mathscr G\skel 0$, then the \emph{orbit} $\mathcal O_x$ of $x$ consists of all $y\in \mathscr G\skel 0$ such that there is an arrow $g$ with $\dom(g)=x$ and $\ran(g)=y$.  The orbits form a partition of $\mathscr G\skel 0$. If $\mathscr G$ is ample, then the orbits of $\mathscr G$ are precisely the orbits for the natural action of $\Bis_c(\mathscr G)$ on $\mathscr G\skel 0$.

A subset $X\subseteq \mathscr G\skel 0$ is \emph{invariant} if it is a union of orbits.
Equivalently, $X$ is invariant if and only if it is invariant under the natural action of $\Bis_c(\mathscr G)$ on $\mathscr G\skel 0$.  

A key example of an \'etale groupoid is that of a groupoid of germs.  Let $S$ be an inverse semigroup acting on a locally compact Hausdorff space $X$.  The groupoid of germs $\mathscr G=S\ltimes X$ is defined as follows. One puts $\mathscr G\skel 0=X$ and $\mathscr G\skel 1=\{(s,x)\in S\times X\mid x\in X_{s^*s}\}/{\sim}$ where $(s,x)\sim (t,y)$ if and only if $x=y$ and there exists $u\leq s,t$ with $x\in X_{u^*u}$. Note that if $S$ is a group, then there are no identifications. The $\sim$-class of an element $(s,x)$ is denoted $[s,x]$.  The topology on $\mathscr G\skel 1$ has basis all sets of the form $(s,U)$ where $U\subseteq X_{s^*s}$ is open and $(s,U) = \{[s,x]\mid x\in U\}$.  
  One puts $\dom([s,x])=x$, $\ran([s,x])=sx$ and defines $[s,ty][t,y]=[st,y]$.  Inversion is given by $[s,x]\inv = [s^*,sx]$.  Note that $(s,X_{s^*s})\in \Bis(S\ltimes X)$ and if $X_{s^*s}$ is compact, then $(s,X_{s^*s})\in \Bis_c(S\ltimes X)$.  Consult~\cite{Exel,Paterson,mygroupoidalgebra} for details.

\subsection{\'Etale groupoid algebras}
Fix now a commutative ring with unit $R$.  The author~\cite{mygroupoidalgebra} associated an $R$-algebra $R\mathscr G$ to each ample groupoid $\mathscr G$ as follows.  We define $R\mathscr G$ to be the $R$-span in $R^{\mathscr G\skel 1}$ of the characteristic functions $\chi_U$ of compact open subsets $U$ of $\mathscr G\skel 1$.  It is shown in~\cite[Proposition~4.3]{mygroupoidalgebra} that $R\mathscr G$ is spanned by the elements $\chi_U$ with $U\in \Bis_c(\mathscr G)$.  If $\mathscr G\skel 1$ is Hausdorff, then $R\mathscr G$ consists of the locally constant $R$-valued functions on $\mathscr G\skel 1$ with compact support.  Convolution is defined on $R\mathscr G$ by \[\p\ast \psi(g)=\sum_{\dom(h)=\dom(g)}\p(gh\inv)\psi(h).\]  The finiteness of this sum is proved in~\cite{mygroupoidalgebra}. The fact that the convolution belongs to $R\mathscr G$ rests on the computation $\chi_U\ast \chi_V=\chi_{UV}$ for $U,V\in \Bis_c(\mathscr G)$~\cite{mygroupoidalgebra}.  Note that $R\mathscr G$ is a quotient of the inverse semigroup algebra $R\Bis_c(\mathscr G)$.




The algebra $R\mathscr G$ is unital if and only if $\mathscr G\skel 0$ is compact, but it always has local units (i.e., is a directed union of unital subrings)~\cite{mygroupoidalgebra,groupoidbundles}.

\section{Topological transitivity of \'etale groupoids}
Primeness of ample groupoid algebras turns out to be closely related to the dynamical property of topological transitivity.  The definition is a straightforward adaptation to groupoids of a topologically transitive group action on a space.  Fix an \'etale groupoid $\mathscr G$.

Let us begin with two elementary propositions.  The first one gives us a large source of open invariant subspaces.

\begin{Prop}\label{p:open.inv}
Let $U\subseteq \mathscr G\skel 0$ be open.  Then $\ran\dom^{-1}(U) = \dom\ran^{-1}(U)$ is open and invariant.  It is, moreover, the smallest invariant subset containing $U$.
\end{Prop}
\begin{proof}
First note that $\ran\dom^{-1}(U)$ is trivially open. Also it is invariant because if $x\in \ran\dom\inv (U)$ and $g\colon x\to y$, then there exists $h\colon z\to x$ with $z\in U$ and so $gh\colon z\to y$ shows that $y\in \ran\dom\inv(U)$. Similarly, $\dom\ran\inv(U)$ is invariant. Obviously, $U\subseteq \ran\dom\inv(U)$ and $\ran\dom\inv(U)$ is contained in any invariant subset containing $U$.  In particular, $\ran\dom\inv(U)\subseteq \dom\ran\inv(U)$.  By symmetry, we obtain the reverse containment.
\end{proof}

The next proposition observes that the interior and the closure of an invariant set are invariant.

\begin{Prop}\label{p:int.closure}
Let $X\subseteq \mathscr G\skel 0$ be invariant.  Then $\mathrm{Int}(X)$ and $\overline{X}$ are also invariant.
\end{Prop}
\begin{proof}
Since $\ov{X} = \mathscr G\skel 0\setminus \mathrm{Int}(\mathscr G\skel 0\setminus X)$ and invariant sets are closed under complementation, it suffices to handle the case of $\mathrm{Int}(X)$.  Suppose $x\in U\subseteq X$ with $U$ open and let $g\colon x\to y$ be an arrow.  Then $y\in \ran\dom^{-1}(U)\subseteq X$, since $X$ is invariant, and $\ran\dom^{-1}(U)$ is open.  Thus $y\in \mathrm{Int}(X)$.
\end{proof}

The following proposition establishes the equivalence of a number of conditions, any of which could then serve as the definition of topological transitivity.

\begin{Prop}\label{p:top.trans}
Let   $\mathscr G$ be an \'etale groupoid.  Then the following are equivalent.
\begin{enumerate}
  \item Every pair of non-empty open invariant subsets of $\mathscr G\skel 0$ has non-empty intersection.
  \item Each non-empty open invariant subset of $\mathscr G\skel 0$ is dense.
  \item Each invariant subset of $\mathscr G\skel 0$ is either dense or nowhere dense.
  \item If $\emptyset\neq U,V\subseteq \mathscr G\skel 0$ are open subsets, then $\dom\inv (U)\cap \ran\inv (V)\neq \emptyset$.
  \item $\mathscr G\skel 0$ is not a union of two proper, closed invariant subsets.
\end{enumerate}
\end{Prop}
\begin{proof}
The first and last item are trivially equivalent by taking complements.  Suppose that (4) holds and let $X$ be an invariant subset.  Then $U=\mathrm{Int}(\ov X)$ is invariant by Proposition~\ref{p:int.closure}.  Suppose $U\neq \emptyset$ and let $V\neq \emptyset$ be any open subset of $\mathscr G\skel 0$.  By (4), there is an element $g\in \dom\inv(U)\cap \ran\inv (V)$.  Then $g\colon x\to y$ with $x\in U\subseteq \ov X$ and $y\in V$.  Since $\ov X$ is invariant, $y\in \ov X$ and so $X\cap V\neq \emptyset$.  We conclude that $X$ is dense.  Trivially, (3) implies (2) since a non-empty open set is not nowhere dense.  Also (2) implies (1) by definition of density.  Assume now that (1) holds and let $\emptyset\neq U,V\subseteq \mathscr G\skel 0$ be open.  Then $\ran\dom\inv(U)$ and $\dom\ran\inv(V)$ are non-empty, open invariant subsets by Proposition~\ref{p:open.inv}.  So, by (1), there is an element $x\in \ran\dom\inv(U)\cap \dom\ran\inv(V)$ and hence there are arrows $h\colon y\to x$ and $g\colon x\to z$ with $y\in U$ and $z\in V$.  Then $gh\in \dom\inv(U)\cap \ran\inv (V)$.  This completes the proof.
\end{proof}

We define an \'etale groupoid $\mathscr G$ to be \emph{topologically transitive} if the equivalent conditions of Proposition~\ref{p:top.trans} hold. Note that if $G$ is a discrete group acting on a locally compact Hausdorff space $X$, then the groupoid $G\ltimes X$ is topologically transitive if and only if the action of $G$ on $X$ is topologically transitive in the usual sense.
Topological transitivity is closely related to the existence of a dense orbit.

\begin{Lemma}\label{l:dense.orbit}
If $\mathscr G$ has a dense orbit, then it is topologically transitive.  If $\mathscr G\skel 0$ is locally compact, Hausdorff and second countable, then the converse holds (in fact, the set of points with dense orbit is co-meager).
\end{Lemma}
\begin{proof}
Assume that $\mathcal O_x$ is dense and let $U\neq \emptyset$ be open and invariant.  Then $U\cap \mathcal O_x\neq \emptyset$ and hence $\mathcal O_x\subseteq U$ by invariance of $U$.  Thus $U$ is dense and hence $\mathscr G$ is topologically transitive.

Suppose now that $\mathscr G$ is topologically transitive and $\mathscr G\skel 0$ is locally compact, Hausdorff and second countable.  Let $\{U_i\}_{\in \mathbb N}$ be a countable base for its topology.  Let $V_i = \ran\dom\inv(U_i)$ and note that $V_i$ is a non-empty, open invariant subset by Proposition~\ref{p:open.inv} and hence dense by topological transitivity.  Then $V=\bigcap_{i=0}^{\infty} V_i$ is dense by the Baire category theorem and  its complement $\mathscr G\skel 0\setminus V$ is meager.  Suppose that $x\in V$.  We claim that $\mathcal O_x$ is dense.  Indeed, $x\in V_i=\ran\dom\inv(U_i)$ and so there is an arrow $g\colon y\to x$ with $y\in U_i$.  Thus $\mathcal O_x\cap U_i\neq \emptyset$.  As the $U_i$ form a basis for the topology, we conclude that $\mathcal O_x$ is dense.
\end{proof}

I previously observed in~\cite{groupoidprimitive} that for non-Hausdorff groupoids, density often needs to be replaced by a more subtle notion that depends on the base commutative ring $R$.  We recall the definition. From now on $\mathscr G$ will be an ample groupoid and $R$ a commutative ring with unit.  A subset $X\subseteq \mathscr G\skel 0$ is said to be \emph{$R$-dense} if, for each $0\neq f\in R\mathscr G$, there is an element $g\in G$ with $f(g)\neq 0$ and $\dom(g)\in X$.  It is shown in~\cite[Prop.~4.2]{groupoidprimitive} that an $R$-dense set is dense and that the converse holds if $\mathscr G$ is Hausdorff.  Moreover, there are examples of dense sets that are not $R$-dense in the non-Hausdorff setting.


\section{Prime and semiprime \'etale groupoid algebras}
Recall that a ring $A$ is \emph{prime} if $IJ=0$ implies $I=0$ or $J=0$ for any ideals $I,J$ of $A$.  It is easy to see that this is equivalent to the condition that if $axb=0$ for all $x\in A$, then $a=0$ or $b=0$. A ring $A$ is \emph{semiprime} if $I^2=0$ implies $I=0$ for an ideal $I$ or, equivalently, $A$ contains no nilpotent ideals~\cite[Prop.~10.16]{LamBook}. At the level of elements, $A$ is semiprime if $axa=0$ for all $x\in A$ implies $a=0$.   Note that any semiprimitive ring $A$ (one with a trivial Jacobson radical) is semiprime.  A commutative ring with unit is prime if and only if it is an integral domain; it is semiprime if and only if it is reduced (i.e., has no nilpotent elements).  The following characterization of prime group rings is due to Connell, see~\cite[Chpt.~10, Sec.~4, Thm.~A]{LamBook}.

\begin{Thm}[Connell]\label{t:connell}
Let $R$ be a commutative ring with unit and $G$ a group.  Then the group algebra $RG$ is prime if and only if $R$ is an integral domain and $G$ has no non-trivial finite normal subgroups.
\end{Thm}

The anologue of Connell's result for semiprimeness is due to Passman, cf.~\cite[Chpt.~10, Sec.~4, Thm.~B]{LamBook}.

\begin{Thm}[Passman]\label{t:passman}
Let $R$ be a commutative ring with unit and $G$ a group.  Then the group algebra $RG$ is semiprime if and only if $R$ is reduced and the order of any finite normal subgroup of $G$ is not a zero-divisor in $R$.
\end{Thm}

Our first result shows that topological transitivity is a necessary condition for an \'etale groupoid algebra to be prime.

\begin{Prop}\label{p:need.trans}
Let $R$ be a commutative ring with unit and let $\mathscr G$ be an ample groupoid.  If $R\mathscr G$ is prime, then $R$ is an integral domain and $\mathscr G$ is topologically transitive.
\end{Prop}
\begin{proof}
Suppose that $R\mathscr G$ is prime and let $a,b\in R$ with $ab=0$.  Let $\emptyset\neq K\subseteq \mathscr G\skel 0$ be compact open. Then $a\chi_K\ast f\ast b\chi_K= ab(\chi_K\ast f\ast \chi_K)=0$ for any $f\in R\mathscr G$.  Thus $a\chi_K=0$ or $b\chi_K=0$ and so we conclude that $a=0$ or $b=0$.  Thus $R$ is an integral domain.

Fix $\emptyset\neq U,V\subseteq \mathscr G\skel 0$ open invariant subsets. Let $\emptyset\neq K\subseteq U$ and $\emptyset\neq K'\subseteq V$ be compact open.   Then we can find $h\in R\mathscr G$ such that $\chi_K\ast h\ast \chi_{K'}\neq 0$.  If $\alpha\in \mathscr G\skel 1$ with $\chi_K\ast h\ast \chi_{K'}(\alpha)\neq 0$, then $\alpha = \alpha_1\alpha_2\alpha_3$ with $\chi_K(\alpha_1)h(\alpha_2)\chi_{K'}(\alpha_3)\neq 0$.  But then $\alpha_1\in K$, $\alpha_3\in K'$ and so $\alpha_2\in \ran\inv(K)\cap \dom\inv(K')\subseteq \ran\inv(U)\cap \dom\inv(V)$. Thus $\mathscr G$ is topologically transitive by Proposition~\ref{p:top.trans}.
\end{proof}

Another necessary condition for primeness comes from the group algebras at isolated points of $\mathscr G\skel 0$.  First we recall that if $A$ is a prime ring, then each corner $eAe$, with $e\neq 0$ an idempotent, is prime.  Indeed, if $a,b\in eAe\setminus \{0\}$ and $axb\neq 0$ with $x\in A$, then $a(exe)b=axb\neq 0$ and $exe\in eAe$.  Similarly, if $A$ is semiprime, then $eAe$ is semiprime.

\begin{Prop}\label{p:isolated}
Let $R\mathscr G$ be a (semi)prime ring.  Then $RG_x$ is a (semi)prime  ring for each isolated point $x\in \mathscr G\skel 0$.
\end{Prop}
\begin{proof}
Observe that $0\neq e=\chi_{\{x\}}$ is an idempotent and $eR\mathscr Ge\cong RG_x$ (cf.~\cite[Prop.~4.7]{groupoidprimitive}).
\end{proof}

Next we characterize primeness for effective Hausdorff groupoids.

\begin{Thm}\label{t:effective.case}
Let $\mathscr G$ be an effective Hausdorff ample groupoid and $R$ a commutative ring with unit.  Then $R\mathscr G$ is prime if and only if $R$ is an integral domain and $\mathscr G$ is topologically transitive.
\end{Thm}
\begin{proof}
The necessity of the conditions follows from Proposition~\ref{p:need.trans}.  For the sufficiency, let $I,J\subseteq R\mathscr G$ be non-zero ideals.  Then, by~\cite[Prop.~3.3]{groupoidprimitive}, there exist $a,b\in R\setminus \{0\}$ and $\emptyset\neq U,V\subseteq \mathscr G\skel 0$ compact open such that $a\chi_U\in I$ and $b\chi_V\in J$.  By Proposition~\ref{p:top.trans}, there is an arrow  $g\in \ran\inv (U)\cap \dom\inv(V)$. Let $K\in \Bis_c(\mathscr G)$ with $g\in K$.  Then $a\chi_U\ast \chi_K\ast b\chi_V = ab\chi_{UKV}\neq 0$ as $R$  an integral domain  implies $ab\neq 0$ and $g\in UKV$.  Thus $IJ\neq 0$ and hence $R\mathscr G$ is prime.
\end{proof}

It was shown in~\cite{groupoidprimitive} that if $\Bbbk$ is a field and $\mathscr G$ is an effective Hausdorff groupoid, then $\Bbbk \mathscr G$ is primitive (i.e., has a faithful irreducible representation) if and only if $\mathscr G$ has a dense orbit.  Primitivity is a stronger notion than primeness, but in light of Theorem~\ref{t:effective.case} and the result just mentioned, we see that they are equivalent for second countable effective Hausdorff groupoids over a field.

\begin{Cor}\label{c:primevsprim}
Let $\mathscr G$ be an effective Hausdorff ample groupoid with $\mathscr G\skel 0$ second countable and $\Bbbk$ a field.  Then $\Bbbk \mathscr G$ is prime if and only if it is primitive.
\end{Cor}
\begin{proof}
By Theorem~\ref{t:effective.case}, $\Bbbk \mathscr G$ is prime if and only if $\mathscr G$ is topologically transitive and by~\cite[Thm.~4.10]{groupoidprimitive} $\Bbbk \mathscr G$ is primitive if and only if $\mathscr G$ has a dense orbit. But when $\mathscr G\skel 0$ is second countable, these are equivalent conditions by Lemma~\ref{l:dense.orbit}.
\end{proof}

An action of a discrete group $G$ on a locally compact Hausdorff space $X$ is said to be \emph{topologically free} if the fixed point set of each non-identity element of $G$ is nowhere dense.  The action groupoid $G\ltimes X$ is well known to be effective if and only if the action is topologically free, cf.~\cite[Prop.~5.6]{groupoidprimitive}. Let $C_c(X,R)$ be the ring of compactly supported, locally constant functions from $X$ to $R$ equipped with the pointwise operations; so $C_c(X,R)$ is the algebra of $X$ viewed as a groupoid of identity morphisms.  Then it is well known that the crossed product algebra $C_c(X,R)\rtimes G$ is the algebra $R[G\ltimes X]$.  Thus Theorem~\ref{t:effective.case}, Corollary~\ref{c:primevsprim} and~\cite[Thm.~4.10]{groupoidprimitive} have the following corollary.

\begin{Cor}
Let $G$ be a discrete group acting on a Hausdorff space $X$ with a basis of compact open sets.  Assume that the action of $G$ is topologically free and let $R$ be a commutative ring with unit. Then  $C_c(X,R)\rtimes G$ is prime if and only if $R$ is an integral domain and the action of $G$ on $X$ is topologically transitive.  If $R$ is a field, then $C_c(X,R)\rtimes G$ is primitive if and only if $G$ has a dense orbit on $X$.  If $X$ is second countable and $R$ is a field, then $C_c(X,R)\rtimes G$ is  prime if and only if it is primitive.
\end{Cor}

Note that it was shown in~\cite{groupoidprimitive} that if $\mathscr G$ is Hausdorff and effective, and $R$ is semiprimitve, then $R\mathscr G$ is semiprimitive.  The corresponding result is also true for semiprimeness.

\begin{Thm}\label{t:semiprime.effective}
Let $R$ be a commutative ring with unit and $\mathscr G$ an effective Hausdorff ample groupoid. Then $R\mathscr G$ is semiprime if and only if $R$ is reduced.
\end{Thm}
\begin{proof}
Necessity is clear since if $r\in R$ with $r^2=0$ and $\emptyset\neq U\subseteq \mathscr G\skel 0$ is compact open, then $r\chi_U\ast f\ast r\chi_U=r^2(\chi_U\ast f\ast\chi_U)=0$ for any $f\in R\mathscr G$ and hence if $R\mathscr G$ is semiprime, then $r=0$.  Thus $R$ is reduced.  Conversely, if $R$ is reduced and $0\neq I$ is an ideal of $R\mathscr G$, then by~\cite[Prop.~3.3]{groupoidprimitive} there exists $a\in R\setminus \{0\}$ and $\emptyset \neq U\subseteq \mathscr G\skel 0$ compact open with $a\chi_U\in I$.  Then $a\chi_U\ast a\chi_U=a^2\chi_U\neq 0$ as $R$ is reduced.  Thus $I^2\neq 0$.  We conclude that $R\mathscr G$ is semiprime.
\end{proof}

Our next result, which is one of the main results of the paper, gives a sufficient condition condition for an ample groupoid algebra to be prime.  In the proof, we will use the well-known fact that if $\mathscr G$ is an ample groupoid with finitely many objects and one orbit, then $R\mathscr G\cong M_n(RG)$ where $n$ is the number of objects of $\mathscr G$ and $G$ is an isotropy group of $\mathscr G$ ($G$ is well defined up to isomorphism), cf.~\cite{gpdchain}.

\begin{Thm}\label{t:dense.orbit.case}
Let $\mathscr G$ be an ample groupoid and $R$ a commutative ring with unit.  Suppose that $\mathscr G$ contains an $R$-dense orbit $\mathcal O_x$ such that $RG_x$ is prime (i.e., $R$ is an integral domain and $G_x$ has no finite non-trivial normal subgroups).  Then $R\mathscr G$ is a prime ring.
\end{Thm}
\begin{proof}
Let $0\neq f,g\in R\mathscr G$.  Since $\mathcal O_x$ is $R$-dense, we can find $\alpha,\beta\in \mathscr G\skel 1$ with $f(\alpha)\neq 0\neq g(\beta)$ and $\ran(\alpha),\dom(\beta)\in \mathcal O_x$.  Choose $\gamma\colon \ran(\alpha)\to x$ and $\gamma'\colon x\to \dom(\beta)$ and let $K,K'\in \Bis_c(\mathscr G)$ with $\gamma\in K$ and $\gamma'\in K'$.  Then $\chi_K\ast f(\gamma\alpha) = f(\alpha)\neq 0$, $g\ast\chi_{K'}(\beta\gamma') = g(\beta)\neq 0$ and it suffices to find $h\in R\mathscr G$ with $(\chi_K\ast f)\ast h\ast (g\ast \chi_{K'})\neq 0$.  Thus, replacing $f$ by $\chi_K\ast f$ and $g$ by $g\ast \chi_{K'}$ we may assume without loss of generality that $\til f=f|_{\ran\inv(x)}\neq 0\neq g|_{\dom\inv(x)}=\til g$.

Since $\ran\inv(x)$ and $\dom\inv(x)$ are closed and discrete, they intersect any compact subset of $\mathscr G\skel 0$ in finitely many points.  Since $f$ and $g$ are finite linear combinations of characteristic functions of compact open subsets, it follows that $\til f,\til g$ are finitely supported.  Let
\[X=\dom(\til f\inv(R\setminus \{0\}))\cup \ran(\til g\inv(R\setminus \{0\}))\cup \{x\}\] and observe that $X$ is finite.  Let  $\mathscr H$ be the subgroupoid of $\mathscr G$ with $\mathscr H\skel 0=X$, $\mathscr H\skel 1 = \mathscr G\skel 1\cap \dom\inv(X)\cap \ran\inv(X)$; notice that it inherits the discrete topology from $\mathscr G$.  We can view $\til f,\til g$ as elements of $R\mathscr H$.  Since $\mathscr H$ is a discrete groupoid with one orbit (by construction each element of $X$ is in the orbit of $x$) and finitely many objects, $R\mathscr H\cong M_{|X|}(RG_x)$ (note that $\mathscr G$ and $\mathscr H$ have the same isotropy group at $x$ by construction).  As $RG_x$ is a prime ring,  $M_{|X|}(RG_x)$ is prime, cf.~\cite[Proposition~10.20]{LamBook}.  Thus we can find $\til h\in R\mathscr H$ with $\til f\ast \til h\ast \til g\neq 0$.

Let $Y=\til h\inv (R\setminus \{0\})$ and note that $Y$ is finite.  For each $\gamma\in Y$, choose $K_{\gamma}\in \Bis_c(\mathscr G)$ such that $\gamma\in K_{\gamma}$.  Since $\mathscr G\skel 0$ is Hausdorff and $X$ is finite, we can find $U_{\gamma},V_{\gamma}\subseteq \mathscr G\skel 0$ compact open such that $U_{\gamma}\cap X=\{\ran(\gamma)\}$ and $V_{\gamma}\cap X=\{\dom(\gamma)\}$.  Using that $K_{\gamma}$ is a local bisection  and replacing $K_{\gamma}$ by $U_{\gamma}K_{\gamma}V_{\gamma}$, we may assume without loss of generality that $K_{\gamma}\cap \mathscr H\skel 1=\{\gamma\}$.  Let us put
\[h = \sum_{\gamma\in Y}\til h(\gamma)\chi_{K_{\gamma}}.\]  Then $h\in R\mathscr G$.  We claim that if $\eta\in (\til f\ast \til h\ast \til g)\inv (R\setminus \{0\})$, then $f\ast h\ast g(\eta)=\til f\ast \til h\ast \til g(\eta)\neq 0$.

Indeed, since $\til f$ is supported on $\ran\inv (x)$ and $\til g$ is supported on $\dom\inv (x)$, we conclude that $\eta\in G_x$.  Note that
\begin{equation}\label{eq:1}
f\ast h\ast g(\eta)=\sum_{\alpha_1\alpha_2\alpha_3=\eta}f(\alpha_1)h(\alpha_2)g(\alpha_3).
\end{equation}
If $\alpha_1\alpha_2\alpha_3=\eta$ with $f(\alpha_1)h(\alpha_2)g(\alpha_3)\neq 0$, then $\ran(\alpha_1)=\ran(\eta)=x$ and $\dom(\alpha_3)=\dom(\eta)=x$ and so $\dom(\alpha_2),\ran(\alpha_2)\in X$ by construction. Therefore, $\alpha_1,\alpha_2,\alpha_3\in \mathscr H\skel 1$.  Moreover, since the support of $h$ is contained in $\bigcup_{\gamma\in Y} K_{\gamma}$ and $\mathscr H\skel 1\cap K_{\gamma}=\{\gamma\}$, for $\gamma\in Y$,  we obtain that $\alpha_2\in Y$ and \[f(\alpha_1)h(\alpha_2)g(\alpha_3) = \til f(\alpha_1)\til h(\alpha_2)\til g(\alpha_3).\]  Conversely, if $\alpha_1,\alpha_2,\alpha_3\in \mathscr H\skel 1$ with $\alpha_1\alpha_2\alpha_3=\eta$, then $\til f(\alpha_1)\til h(\alpha_2)\til g(\alpha_3)=f(\alpha_1)h(\alpha_2)g(\alpha_3)$ by construction.  It follows  that the right hand side of  \eqref{eq:1} is
precisely $\til f\ast \til h\ast \til g(\eta)$.  This completes the proof.
\end{proof}

Next we prove an analogue of Theorem~\ref{t:dense.orbit.case} for semiprimeness.

\begin{Thm}\label{t:dense.semiprime}
Let $\mathscr G$ be an ample groupoid and $R$ a commutative ring with unit.  Suppose that the set $Z$ of $x\in \mathscr G\skel 0$ such that $RG_x$ is semiprime is $R$-dense.  Then $R\mathscr G$ is a semiprime ring.
\end{Thm}
\begin{proof}
Let $0\neq f\in R\mathscr G$. We need to show that $f\ast h\ast f\neq 0$ for some $h\in R\mathscr G$.  Observe that the set $Z$ is invariant since elements of the same orbit have isomorphic isotropy groups. Since $Z$ is $R$-dense, we can find $\alpha\in \mathscr G\skel 1$ with $f(\alpha)\neq 0$ and $x=\dom(\alpha)\in Z$.  Let $K\in \Bis_c(\mathscr G)$ with $\alpha\inv\in K$.  Then $\chi_K\ast f(x) = f(\alpha)\neq 0$ and it suffices to find $h\in R\mathscr G$ such that $(\chi_K\ast f)\ast h\ast (\chi_K\ast f)\neq 0$.  Thus, replacing $f$ by $\chi_K\ast f$  we may assume without loss of generality that $\til f=f|_{G_x}\neq 0$.

Note that $\til f$ is finitely supported since $G_x$ is closed and discrete.  Let
\[X=\dom(f|_{\ran\inv(x)}\inv(R\setminus \{0\}))\cup \ran(f|_{\dom\inv(x)}\inv(R\setminus \{0\}))\cup \{x\}\] and observe that $X$ is finite since the fibers of $\dom$ and $\ran$ are closed and discrete.   We  view $\til f$ as an element of $RG_x$. Since $RG_x$ is semiprime, we can find $\til h\in RG_x$ with $\til f\ast \til h\ast\til f\neq 0$.

Let $Y=\til h\inv (R\setminus \{0\})$ and note that $Y$ is finite.  For each $\gamma\in Y$, choose $K_{\gamma}\in \Bis_c(\mathscr G)$ such that $\gamma\in K_{\gamma}$.  Since $\mathscr G\skel 0$ is Hausdorff and $X$ is finite, we can find $U\subseteq \mathscr G\skel 0$ compact open such that $U\cap X=\{x\}$.  Using that $K_{\gamma}$ is a local bisection and replacing $K_{\gamma}$ by $UK_{\gamma}U$, we may assume without loss of generality that $K_{\gamma}\cap \dom\inv(X)\cap \ran\inv(X)=\{\gamma\}$.  Set
\[h = \sum_{\gamma\in Y}\til h(\gamma)\chi_{K_{\gamma}}.\]  Then $h\in R\mathscr G$.  We claim that if $\eta\in (\til f\ast \til h\ast \til f)\inv (R\setminus \{0\})$, then $f\ast h\ast f(\eta)=\til f\ast \til h\ast \til f(\eta)\neq 0$.

Observe that
\begin{equation}\label{eq:2}
f\ast h\ast f(\eta)=\sum_{\alpha_1\alpha_2\alpha_3=\eta}f(\alpha_1)h(\alpha_2)f(\alpha_3).
\end{equation}
If $\alpha_1\alpha_2\alpha_3=\eta$ with $f(\alpha_1)h(\alpha_2)f(\alpha_3)\neq 0$, then $\ran(\alpha_1)=\ran(\eta)=x$ and $\dom(\alpha_3)=\dom(\eta)=x$, whence $\dom(\alpha_2),\ran(\alpha_2)\in X$ by construction.  Moreover, since the support of $h$ is contained in $\bigcup_{\gamma\in Y} K_{\gamma}$, by construction $\alpha_2\in G_x$ and hence \[f(\alpha_1)h(\alpha_2)f(\alpha_3) = \til f(\alpha_1)\til h(\alpha_2)\til f(\alpha_3).\]  Conversely, if $\eta=\alpha_1\alpha_2\alpha_3$ with $\alpha_1,\alpha_2,\alpha_3\in G_x$, then $\til f(\alpha_1)\til h(\alpha_2)\til f(\alpha_3)=f(\alpha_1)h(\alpha_2)f(\alpha_3)$.  It  follows  that the right hand side of  \eqref{eq:2} is
precisely $\til f\ast \til h\ast \til f(\eta)$.  This completes the proof.
\end{proof}

\section{Applications}
In this section, we show that the results of the previous section are strong enough to recover the results of Munn~\cite{Munnprime} on primeness of inverse semigroup algebras and of Abrams, Bell and Rangaswamy on prime Leavitt path algebras~\cite{primleavittbell}.

\subsection{Leavitt path algebras}
To recover the description of prime Leavitt path algebras from~\cite{primleavittbell}, we need to combine Theorem~\ref{t:effective.case} with Theorem~\ref{t:dense.orbit.case}.
Let $E=(E\skel 0,E\skel 1)$ be a (directed) graph (or quiver) with vertex set $E\skel 0$ and edge set $E\skel 1$.  We use $\sour(e)$ for the source of an edge $e$ and $\ran(e)$ for the range, or target.   A vertex $v$ is called a \emph{sink} if $\sour\inv(v)=\emptyset$ and it is called an \emph{infinite emitter} if $|\sour\inv(v)|=\infty$.  The length of a finite (directed) path $\alpha$ is denoted $|\alpha|$.

The \emph{Leavitt path algebra}~\cite{Leavitt1,LeavittPardo,Abramssurvey,LeavittBook} $L_R(E)$ of $E$ with coefficients in the unital commutative ring $R$ is the $R$-algebra generated by a set $\{v\in E\skel 0\}$ of pairwise orthogonal idempotents and a set of variables $\{e,e^*\mid e\in E\skel 1\}$ satisfying the relations:
\begin{enumerate}
\item $\sour(e)e=e=e\ran(e)$ for all $e\in E\skel 1$;
\item $\ran(e)e^*= e^*=e^*\sour(e)$ for all $e\in E\skel 1$;
\item $e^*e'=\delta_{e,e'}\ran(e)$ for all $e,e'\in E\skel 1$;
\item $v=\sum_{e\in \sour^{-1}(v)} ee^*$ whenever $v$ is not a sink and not an infinite emitter.
\end{enumerate}

It is well known that $L_R(E)=R\mathscr G_E$  for the graph groupoid $\mathscr G_E$ defined as follows.  Let $\partial E$ consist of all one-sided infinite paths in $E$ as well as all finite paths $\alpha$ ending in a vertex $v$ that is either a sink or an infinite emitter.  If $\alpha$ is a finite path in $E$ (possibly empty), put $Z(\alpha)=\{\alpha\beta\in \partial E\}$ (if $\alpha$ is the empty path $\varepsilon_v$ at $v$, this should be interpreted as those elements of $\partial E$ with initial vertex $v$).  Note that $Z(\alpha)$ is never empty.  Then a basic open neighborhood  of $\partial E$ is of the form $Z(\alpha)\setminus (Z(\alpha e_1)\cup\cdots \cup Z(\alpha e_n))$ with $e_i\in E\skel 1$, for $i=1,\ldots n$ (and possibly $n=0$).  These neighborhoods are compact open.

  The graph groupoid $\mathscr G_E$ is the given by:
\begin{itemize}
\item $\mathscr G_E\skel 0= \partial E$;
\item $\mathscr G_E\skel 1 =\{ (\alpha\gamma,|\alpha|-|\beta|,\beta\gamma)\in \partial E\times \mathbb Z\times \partial E\}\mid |\alpha|,|\beta|<\infty\}$.
\end{itemize}
One has $\dom(\eta,k,\gamma)=\gamma$, $\ran(\eta,k,\gamma)=\eta$ and $(\eta,k,\gamma)(\gamma,m,\xi) = (\eta, k+m,\xi)$.  The inverse of $(\eta,k,\gamma)$ is $(\gamma,-k,\eta)$.

A basis of compact open subsets for the topology on $\mathscr G_E\skel 1$ can be described as follows. Let $\alpha,\beta$ be finite paths ending at the same vertex and let $U\subseteq Z(\alpha)$, $V\subseteq Z(\beta)$ be compact open with $\alpha\gamma\in U$ if and only if $\beta\gamma\in V$.  Then the set \[(U,\alpha,\beta,V)=\{\alpha\gamma,|\alpha|-|\beta|,\beta\gamma)\mid \alpha\gamma\in U,\beta\gamma\in V\}\] is a basic compact open set of $\mathscr G_E\skel 1$.
Of particular importance are the compact open sets $Z(\alpha,\beta) = (Z(\alpha),\alpha,\beta,Z(\beta))= \{(\alpha\gamma,|\alpha|-|\beta|,\beta\gamma)\in \mathscr G\skel 1\}$ where $\alpha,\beta$ are finite paths ending at the same vertex.
It is well known, and easy to see, that $\mathscr G_E$ is Hausdorff.

There is an isomorphism $L_R(E)\to R\mathscr G_E$ sending $v\in E\skel 0$ to the characteristic function of $Z(\varepsilon_v,\varepsilon_v)$ and, for $e\in E\skel 1$, sending $e$ to the characteristic function of $Z(e,\varepsilon_{\ran(e)})$ and $e^*$ to the characteristic function of $Z(\varepsilon_{\ran(e)},e)$, cf.~\cite{operatorguys1,groupoidprimitive,CRS17,Strongeffective} or~\cite[Example~3.2]{GroupoidMorita}.

By a \emph{cycle} in a directed graph $E$, we mean a simple, directed,  closed circuit.  A cycle is said to have an \emph{exit} if some vertex on the cycle has out-degree at least two.  It is well known that the isotropy group at an element $\gamma\in \partial E$ is trivial unless $\gamma$ is \emph{eventually periodic}, that is, $\gamma=\rho\alpha\alpha\cdots$ with $\alpha$ a cycle, in which case the isotropy group is infinite cyclic, cf.~\cite{gpdchain}.

If $u,v\in E\skel 0$, we write $u\geq v$ if there is a path (possibly empty) from $u$ to $v$.
The graph $E$ is said to be \emph{downward directed} (or to satisfy condition (MT3), cf.~\cite{primleavittbell,LeavittBook}) if, for each pair of vertices $u,v\in E\skel 0$, there is a vertex $w$ such that $u,v\geq w$.  Our first goal is to verify that condition (MT3) is satisfied if and only if $\mathscr G_E$ is topologically transitive.

\begin{Prop}\label{p:MT3}
The graph $E$ is downward directed if and only if $\mathscr G_E$ is topologically transitive.
\end{Prop}
\begin{proof}
Notice that $\alpha,\beta\in \partial E$ belong to the same orbit if and only if they have a common suffix.  Suppose first that $\mathscr G_E$ is topologically transitive and let $u,v\in E\skel 0$. Then $\ran\inv(Z(\varepsilon_u))\cap \dom\inv(Z(\varepsilon_v))\neq \emptyset$ by Proposition~\ref{p:top.trans}. Hence we can find $(\alpha\gamma,|\alpha|-|\beta|,\beta\gamma)\in \mathscr G_E\skel 1$ with $\beta\gamma\in Z(\varepsilon_v)$ and $\alpha\gamma\in Z(\varepsilon_u)$.  Then if $w$ is the initial vertex of $\gamma$, we have that $\alpha,\beta$ are paths from $u,v$ to $w$, respectively.   Thus $E$ is downward directed.

Suppose now that $E$ is downward directed. Suppose first that $E$ has a sink $w$. Then since $E$ is downward directed, it follows $v\geq w$ for all $v\in E\skel 0$ and in particular $w$ is the unique sink.  Then the orbit of $\varepsilon_w$ is dense.  Indeed, if $\emptyset\neq V= Z(\alpha)\setminus (Z(\alpha e_1)\cup \cdots\cup Z(\alpha e_n))$, then either $\alpha$ ends at $w$ and so $\alpha\in V\cap \mathcal O_{\varepsilon_w}$ or there is an edge $e\neq e_1,\ldots, e_n$ with $\alpha e$ a path. Then there is a path $\beta$ from $\ran(e)$ to $w$ and $\alpha e\beta\in V\cap \mathcal O_{\varepsilon_w}$.  Thus $\mathcal O_{\varepsilon_w}$ is dense and hence $\mathscr G_E$ is topologically transitive by Proposition~\ref{l:dense.orbit}.

Assume now that $E$ does not have a sink.
Let $\emptyset\neq U\subseteq \partial E$ be open and invariant.  Let $\emptyset\neq V= Z(\alpha)\setminus (Z(\alpha e_1)\cup \cdots\cup Z(\alpha e_n))$ be a basic neighborhood and let $\emptyset\neq W=Z(\beta)\setminus (Z(\beta f_1)\cup\cdots\cup Z(\beta f_m))$ be a basic neighborhood contained in $U$.  Then since $V$ and $W$ are non-empty, there are edges $e,f$ with $\alpha e$ and $\beta f$ paths and $e\neq e_i$ and $f\neq f_j$ for all $i,j$.  Then, by the downward directed property, we can find $w\in E\skel 0$ with $\ran(e),\ran(f)\geq w$.  Let $\gamma\in \partial E$ begin at $w$.  Then $\alpha e\rho\gamma, \beta f\sigma\gamma\in \partial E$ for some paths $\rho,\sigma$.  Then $\alpha e\rho\gamma\in V$ and $\beta f\sigma\gamma\in W\subseteq U$ and hence $\alpha e\rho\gamma \in U\cap V$, as $U$ is invariant.   Therefore, $U$ is dense and so $\mathscr G_E$ is topologically transitive in this case, as well.
\end{proof}

The graph $E$ is said to satisfy \emph{condition (L)} if every cycle has an exit.  It is well known that $\mathscr G_E$ is effective if and only if each cycle has an exit, cf.~\cite{Strongeffective}.  Since many of the references assume that $E$ is countable or row-finite, we shall prove it here.  Note that if a Hausdorff \'etale groupoid has a dense set of objects with trivial isotropy groups, then it is effective. Indeed, if $U$ is an open subset contained in the isotropy bundle, then $V=U\setminus \mathscr G\skel 0$ is open (as $\mathscr G\skel 0$ is closed). If $V\neq\emptyset$, then there exists $x\in \dom(V)$ with trivial isotropy group.  But then if $g\in V$ with $\dom(g)=x$, we have $g\in V\cap \mathscr G\skel 0$, a contradiction.  Thus $\mathscr G\skel 0$ is the interior of the isotropy bundle.

\begin{Prop}\label{p:condL}
The graph $E$ satisfies condition (L) if and only if $\mathscr G_E$ is effective.
\end{Prop}
\begin{proof}
Suppose that $E$ does not satisfy condition (L) and that $\alpha$ is a cycle with no exit.  In particular, no vertex of $\alpha$ is a sink or infinite emitter.  Then $Z(\alpha,\alpha)$ contains precisely the elements of the form $(\alpha\alpha\cdots, k|\alpha|,\alpha\alpha\cdots)\in G_{\alpha\alpha\cdots}$ with $k\in \mathbb Z$ and hence $\mathscr G_E$ is not effective.  Conversely, assume that $E$ satisfies condition (L). Note that since every cycle has an exit, each basic open subset of $\partial E$ contains an element $\tau$ that is not eventually periodic. As the non-eventually periodic elements have trivial isotropy and $\mathscr G_E$ is Hausdorff, it follows that $\mathscr G_E$ is effective.
\end{proof}

We now prove the following theorem from~\cite{primleavittbell}.

\begin{Thm}
Let $E$ be a graph and $R$ be a commutative ring with unit.  Then $L_R(E)$ is a prime ring if and only if $R$ is an integral domain and $E$ is downward directed.
\end{Thm}
\begin{proof}
If $L_R(E)$ is a prime ring, then $R$ is an integral domain and $E$ is downward directed by Proposition~\ref{p:need.trans} and Proposition~\ref{p:MT3}.  Suppose, conversely, that $R$ is an integral domain and $E$ is downward directed.  There are two cases.  Suppose first that $E$ satisfies condition (L).  Then $\mathscr G_E$ is a Hausdorff, effective and topologically transitive groupoid by Proposition~\ref{p:MT3} and Proposition~\ref{p:condL}.  Thus $L_R(E)$ is prime by Theorem~\ref{t:effective.case}. On the other hand, suppose that $E$ contains a cycle $\alpha$ with no exit.  Let $u\in E\skel 0$ and $v$ be a vertex of $\alpha$.  Then, since $E$ is downward directed, there exists $w\in E\skel 0$ with $u,v\geq w$. But $w$ is a vertex of $\alpha$ because $\alpha$ has no exit.  Thus every vertex $u$  has a path to a vertex of $\alpha$. In particular, $E$ has no sinks and if $\beta$ is any finite path in $E$, then $\beta\tau\alpha\alpha\cdots\in \partial E$ for some path $\tau$.  It then follows from the definition of the topology on $\partial E$ that the orbit of $\alpha\alpha\cdots$ is dense. (Recall that orbit of $\gamma\in \partial E$ consists of all strings with a common suffix with $\gamma$.) But the isotropy group at $\alpha\alpha\cdots$ is infinite cyclic and so $RG_{\alpha\alpha\cdots}\cong R[x,x\inv]$ is an integral domain and hence a prime ring.  Therefore, $L_R(E)$ is a prime ring by Theorem~\ref{t:dense.orbit.case}.  This completes the proof.
\end{proof}

Note that since Laurent polynomial rings over reduced rings are reduced, we have the following corollary of Theorem~\ref{t:dense.semiprime}.

\begin{Cor}
Let $E$ be a graph and $R$ a commutative ring with unit.  Then $L_R(E)$ is semiprime if and only if $R$ is reduced.
\end{Cor}

The following example is from~\cite{primleavittbell}. Let $X$ be an uncountable set and let $E$ be the graph whose vertices are all finite subsets of $X$ and there is an edge from $A$ to $B$ if $A$ is a proper subset of $B$.  Note that all vertices are infinite emitters and there are no sinks.  Thus every finite and infinite path belongs to $\partial E$.  This graph is clearly downward directed since there is a path (possibly empty) from $A$ to $B$ if and only if $A$ is a subset of $B$ and we can take unions of finite sets.  Therefore, $L_R(E)$ is prime for any integral domain.  However, $\mathscr G_E$ has no dense orbit.  Indeed, if  $\gamma\in \partial E$ and $Y$ is the set of elements of $X$ that appear in some vertex of $\gamma$, then $Y$ is countable (being a countable union of finite sets).  So there exists $x\in X\setminus Y$.  But then $Z(\varepsilon_{\{x\}})$ is an open set missing the orbit of elements with suffix $\gamma$ (as the vertices of an element of $\partial E$ form an increasing chain).  It follows that, for any field $\Bbbk$, we have that $L_{\Bbbk}(E)$ is prime but not primitive (since primitivity for a groupoid algebra implies the existence of a dense orbit by~\cite[Prop.~4.9]{groupoidprimitive}).  This was already observed in~\cite{primleavittbell} via a groupoid-free argument.  As $E$ has no cycles, and hence satisfies condition (L), we see that $\mathscr G_E$ is an effective Hausdorff groupoid that is topologically transitive with no dense orbit and hence second countability really is required in Corollary~\ref{c:primevsprim}.

We shall present here a proof of the primitivity criterion for Leavitt path algebras from~\cite{primleavittbell} using the results of~\cite{groupoidprimitive}, as we neglected to do so in that paper (we just said it was straightforward to do so).   A graph $E$ is said to satisfy the countable separation property (CSP) if there is a countable (finite or countably infinite) set of vertices $X$ such that, for all $v\in E\skel 0$, there exists $x\in X$ with $v\geq x$. The result of~\cite{primleavittbell} is the following, which we prove using groupoids.

\begin{Thm}
Let $E$ be a directed graph and $\Bbbk$ a field.  Then $L_{\Bbbk}(E)$ is primitive if and only if:
\begin{enumerate}
  \item $E$ satisfies condition (L);
  \item $E$ is downward directed;
  \item $E$ has the countable separation property.
\end{enumerate}
\end{Thm}
\begin{proof}
Assume first that $L_{\Bbbk}(E)\cong \Bbbk \mathscr G_E$ is primitive.  Then $\Bbbk \mathscr G_E$ is prime and hence $\mathscr G_E$ is topologically transitive by Proposition~\ref{p:need.trans}.  Thus $E$ is downward directed by Proposition~\ref{p:MT3}.  If $\alpha$ is a cycle in $E$ without an exit, then $Z(\alpha)$ is a neighborhood of $\gamma=\alpha\alpha\cdots$ containing no other element of $\partial E$ and so $\gamma$ is isolated.  Therefore $\Bbbk G_{\gamma}$ is primitive by~\cite[Prop.~4.7]{groupoidprimitive}.  But $G_{\gamma}\cong \mathbb Z$~\cite{gpdchain} and so $\Bbbk G_{\gamma}\cong \Bbbk[x,x\inv]$, which is not primitive.  We conclude that $E$ satisfies condition (L). Thus $\mathscr G_E$ is an effective Hausdorff groupoid and hence has a dense orbit by~\cite[Thm.~4.10]{groupoidprimitive}.  Let $\gamma\in \partial E$ belong to this dense orbit and let $X$ be the set of vertices of $\gamma$.  Note that $X$ is countable.  If $v\in E\skel 0$, then $Z(\varepsilon_v)$ intersects the orbit of $\Gamma$.  Hence we can find $\alpha,\beta$ finite paths such that $\gamma=\beta\tau$ and $\alpha\tau \in Z(\varepsilon_v)$.  But then $\alpha$ is a path from $v$ to a vertex of $\gamma$.  Thus $E$ has the countable separation property.

Conversely, suppose that $E$ satisfies condition (L), is downward directed and has CSP.  Then $\mathscr G_E$ is effective and Hausdorff and so it suffices to show that it has a dense orbit by~\cite[Thm.~4.10]{groupoidprimitive}.  Let $X$ be a countable set as in the definition of the countable separation property.  First assume that $X$ is finite.  Then since $E$ is downward directed, we can find $v$ with $x\geq v$ for all $x\in X$.  By definition of CSP, it follows that $w\geq v$ for all $w\in E\skel 0$.  Choose $\gamma\in \partial E$ originating from $v$.  Then any finite path $\alpha$ can be continued to a path ending at $v$ and hence $\alpha\beta\gamma\in \partial E$ for some finite path $\beta$.  It follows from the definition of the topology on $\partial E$ that the orbit of $\gamma$ is dense.  Next suppose that $X=\{v_1,v_2,\ldots\}$ is countably infinite.  Set $x_1=v_1$ and assume inductively we have chosen $x_1,\ldots, x_n$ such that $x_1\geq x_2\geq\cdots \geq x_n$ and $v_i\geq x_i$ for $1\leq i\leq n$.  Then since $E$ is downward directed, we can find $x_{n+1}$ with $x_1,\ldots,x_n,v_{n+1}\geq x_{n+1}$.  Thus we can find an infinite path $\gamma=\alpha_1\alpha_2\cdots$ where $\alpha_i$ is a directed path from $x_i$ to $x_{i+1}$.  We claim that the orbit of $\gamma$ is dense.  Again, any finite path $\alpha$ has a continuation to a vertex $v_n$ of $X$ and hence, by construction, to a vertex $x_n$ of $\gamma$.  Hence, any finite path has a continuation to an infinite path having a common suffix with $\gamma$. It follows from the definition of the topology that the orbit of $\gamma$ is dense.  This completes the proof.
\end{proof}

\subsection{Inverse semigroup algebras}
To apply the above results  to inverse semigroups, we need to discuss first how to realize an inverse semigroup algebra as an ample groupoid algebra.  Fix an  inverse semigroup $S$ for the rest of this subsection.
First we recall the construction of the universal groupoid $\mathscr G(S)$ of an inverse semigroup and the contracted universal groupoid $\mathscr G_0(S)$ for an inverse semigroup with zero. See~\cite{Exel,Paterson,mygroupoidalgebra,groupoidprimitive} for details.

A \emph{character} of a semilattice $E$ is a non-zero homomorphism $\theta\colon E\to \{0,1\}$ where $\{0,1\}$ is a semilattice under multiplication.  The \emph{spectrum} of $E$ is the space $\wh E$ of characters of $E$, topologized as a subspace of $\{0,1\}^E$.  Note that $\wh E$ is Hausdorff with a basis of compact open sets.  Indeed, if we put $D(e)=\{\theta\in \wh E\mid \theta(e)=1\}$ for $e\in E(S)$, then the sets of the form $D(e)\cap D(e_1)^c\cap\cdots D(e_n)^c$ form a basis of compact open sets for the topology, where $X^c$ denotes the complement of $X$.   If $e\in E$, then the \emph{principal character} $\theta_e\colon E\to \{0,1\}$ is defined by
\[\theta_e(f)=\begin{cases} 1, & \text{if}\ f\geq e\\ 0, & \text{else.}\end{cases}\]  The principal characters are dense in $\wh E$.  If $E$ has a zero element, then a character $\theta$ is called \emph{proper} if $\theta(0)=0$ or, equivalently, $\theta\neq \theta_0$.  The set of proper characters will be denoted $\wh E_0$.  Notice that $D(0)=\{\theta_0\}$ and so $\theta_0$ is always an isolated point of $\wh E$.

Let $S$ be an inverse semigroup. Then $S$ acts on $\wh{E(S)}$.  The domain of the action of $s$ is $D(s^*s)$.  If $\theta\in D(s^*s)$, then $(s\theta)(e) = \theta(s^*es)$.   If $S$ has a zero, then $\wh{E(S)}_0$ is invariant under $S$.  The \emph{universal groupoid} of $S$ is the groupoid of germs $\mathscr G(S)=S\ltimes \wh{E(S)}$. Note that the isotropy group $G_{\theta_e}$ of a principal character $\theta_e$ is isomorphic to the maximal subgroup $G_e$  and two principal characters $\theta_e,\theta_f$ are in the same orbit if and only if there exists $s\in S$ with $s^*s=e$ and $ss^*=f$ (cf.~\cite{mygroupoidalgebra}).

If $S$ has a zero, we put $\mathscr G_0(S)=S\ltimes \wh{E(S)}_0$ and call it the \emph{contracted universal groupoid} of $S$.

The following theorem is fundamental to the subject.  See~\cite{mygroupoidalgebra,groupoidprimitive}.

\begin{Thm}\label{t:isothm}
Let $S$ be an inverse semigroup and $R$ a commutative ring with unit.  Then $RS\cong R\mathscr G(S)$. The isomorphism sends $s\in S$ to $\chi_{(s,D(s^*s))}$.  If $S$ has a zero, then $R_0S\cong R\mathscr G_0(S)$.
\end{Thm}

A semilattice $E$ is \emph{pseudofinite}~\cite{MunnSemiprim} if, for all $e\in E$, the set of elements strictly below $e$ is finitely generated as a lower set.  In~\cite[Prop.~2.5]{mygroupoidalgebra}, it was shown that this is equivalent to the principal characters being isolated points of $\wh{E}$.  


It is proved in~\cite{groupoidprimitive} that the principal characters are $R$-dense in $\wh{E(S)}$ for any commutative ring with unit $R$.  We recall that an inverse semigroup $S$ (with zero) is said to be \emph{($0$-)bisimple} if, for all $e,f\in E(S)$ (non-zero), there exists $s\in S$ with $s^*s=e$ and $ss^*=f$.  In a ($0$-)bisimple inverse semigroup, all (non-zero) idempotents have isomorphic maximal subgroups, cf.~\cite{Lawson}. The following result was originally proved by Munn~\cite{Munnprime}.

\begin{Thm}\label{t:munn.prime}
Let $S$ be an inverse semigroup (with zero). Then if $S$ is ($0$-)bisimple with maximal subgroup $G$ (at a non-zero idempotent) and $R$  is a commutative ring with unit such that $RG$ is prime, then $RS$ ($R_0S$) is prime.  The converse holds if $E(S)$ is pseudofinite.
\end{Thm}
\begin{proof}
We just handle the case without zero,  as the case of semigroups with zero is identical.  The hypothesis that $S$ is bisimple is equivalent to the principal characters forming a single orbit.   Thus the principal characters form an $R$-dense orbit with isotropy group $G$ and the first statement follows from Theorem~\ref{t:dense.orbit.case}.

Conversely, suppose that $E(S)$ is pseudofinite and $RS$ is prime.  Then each principal character is isolated and the orbit of a principal character is an open invariant set.  Since $\mathscr G(S)$ must be topologically transitive by Proposition~\ref{p:need.trans}, we deduce that there is only one orbit of principal characters.  But this is equivalent to $S$ being bisimple.  Moreover, as each principal character is isolated, if $G$ is the maximal subgroup of $S$, then $RG$ is prime by Proposition~\ref{p:isolated}.
\end{proof}

Let us consider the analogue for semiprimeness.  The following result is due to Munn~\cite{Munnprime}.

\begin{Thm}
Let $S$ be an inverse semigroup (with zero) and $R$ a commutative ring with unit.  If $RG_e$ is semiprime for each idempotent $e$, then $RS$ ($R_0S$) is semiprime.  The converse holds if $E(S)$ is pseudofinite.
\end{Thm}
\begin{proof}
Since the principal characters are $R$-dense the sufficiency follows from Theorem~\ref{t:dense.semiprime}.  If $E(S)$ is pseudofinite, then the principal characters are isolated and Proposition~\ref{p:isolated} provides the desired conclusion.
\end{proof}

We remark that it is a long-standing, and most likely difficult, question to describe all prime or semiprime inverse semigroup algebras.

\end{document}